\documentclass[10pt,reqno]{amsart}
\usepackage{latexsym,amsmath}
\usepackage{times}
\usepackage{amsfonts}
\usepackage{amssymb}
\usepackage{latexsym}
\usepackage{enumerate}

\usepackage{bbm}

\newcommand{\beq}{\begin{equation}} \newcommand{\eeq}{\end{equation}}

\numberwithin{equation}{section}

\newcommand\br[1]{\left(#1\right)}

\def\vec#1{\mathchoice{\mbox{\boldmath$\displaystyle#1$}}
{\mbox{\boldmath$\textstyle#1$}}
{\mbox{\boldmath$\scriptstyle#1$}}
{\mbox{\boldmath$\scriptscriptstyle#1$}}}

\DeclareMathOperator{\pr}{\mathbb P}

\newcommand\SIGMA{\vec\sigma}

\newtheorem{definition}{Definition}[section]

\newtheorem{theorem}[definition]{Theorem}
\newtheorem{lemma}[definition]{Lemma}
\newtheorem{proposition}[definition]{Proposition}
\newtheorem{corollary}[definition]{Corollary}

\newtheorem{algorithm}[definition]{Algorithm}
\newtheorem{fact}[definition]{Fact}

\usepackage[dvips]{epsfig}
\graphicspath{{./Fig_eps/}{./Eps/}}
\DeclareGraphicsExtensions{.ps,.eps}
\usepackage{color}
\usepackage{fullpage}

\newcommand\rk{r_{k\mathrm{-SAT}}}

\newcommand\dist{\mbox{dist}}

\newcommand\suc{\mathrm{success}}

\newcommand\PHI{\vec\Phi}

\newcommand\cA{\mathcal{A}}

\newcommand\cC{\mathcal{C}}
\newcommand\cD{\mathcal{D}}

\newcommand\cG{\mathcal{G}}
\newcommand\cE{\mathcal{E}}
\newcommand\cU{\mathcal{U}}

\newcommand\cH{\mathcal{H}}

\newcommand\cM{\mathcal{M}}

\newcommand\cP{\mathcal{P}}

\def\cC{{\mathcal C}}
\def\cE{{\mathcal E}}

\newcommand\eul{\mathrm{e}}

\newcommand\Erw{\mathbb{E}}

\newcommand{\vecone}{\vec{1}}

\newcommand{\Bin}{{\rm Bin}}

\newcommand{\bink}[2] {{{#1}\choose {#2}}}

\newcommand\bc[1]{\left({#1}\right)}
\newcommand\cbc[1]{\left\{{#1}\right\}}
\newcommand\bcfr[2]{\bc{\frac{#1}{#2}}}

\newcommand\brk[1]{\left\lbrack{#1}\right\rbrack}

\newcommand\abs[1]{\left|{#1}\right|}

\newcommand{\whp}{w.h.p.}

\newcommand\vphi{{\vec{\varPhi}}}

\newcommand\Lem{Lemma}
\newcommand\Prop{Proposition}
\newcommand\Thm{Theorem}
\newcommand\Cor{Corollary}
\newcommand\Sec{Section}

\usepackage{courier}
\newcommand\algstyle{\small\sffamily}
\newcommand{\algn}[1]{\textnormal{\texttt{#1}}}

\newcommand{\walksat}{\algn{Walksat}}

\newcommand\KL[2]{D_{\mathrm{KL}}\bc{{{#1},{#2}}}}

\begin{document}

\title{
	Walksat stalls well below the satisfiability threshold}

\author[]{Amin Coja-Oghlan, Amir Haqshenas and Samuel Hetterich}
\thanks{$^\star$ The research leading to these results has received funding from the European Research Council under the European Union's Seventh Framework
			Programme (FP/2007-2013) / ERC Grant Agreement n.\ 278857--PTCC}

\address{Amin Coja-Oghlan, {\tt acoghlan@math.uni-frankfurt.de}, Goethe University, Mathematics Institute, 10 Robert Mayer St, Frankfurt 60325, Germany.}

\address{Amir Haqshenas, {\tt amirhaqshenas@googlemail.com}, Goethe University, Mathematics Institute, 10 Robert Mayer St, Frankfurt 60325, Germany.}

\address{Samuel Hetterich, {\tt hetterich@math.uni-frankfurt.de}, Goethe University, Mathematics Institute, 10 Robert Mayer St, Frankfurt 60325, Germany.}

\maketitle

\begin{abstract}
\noindent
Partly on the basis of heuristic arguments from physics it has been suggested that the performance of certain types of algorithms on random $k$-SAT formulas is linked
to phase transitions that affect the geometry of the set of satisfying assignments.
But beyond intuition there has been scant rigorous evidence that ``practical'' algorithms are affected by these phase transitions.
In this paper we prove that \walksat, a popular randomised satisfiability algorithm, fails on random $k$-SAT formulas not very far above clause/variable density where the set of satisfying assignments shatters into tiny, well-separated clusters.
Specifically, we prove \walksat\ is ineffective with high probability
if $m/n>c2^k\ln^2k/k$, where $m$ is the number of clauses, $n$ is the number of variables and $c>0$ is an absolute constant.
By comparison, \walksat\ is known to find satisfying assignments in linear time \whp\ if $m/n<c'2^k/k$ for another constant $c'>0$
	[Coja-Oghlan and Frieze, SIAM J.\ Computing 2014].

\end{abstract}

\section{Introduction and results}

\subsection{Background and motivation.}
For integers $k\geq 3$ and $n,m>0$ let $\vphi=\vphi_k(n,m)=\vphi_1 \wedge \ldots \wedge \vphi_m$
	be a random Boolean formula in conjunctive normal form with clauses
	$\vphi_i=\vphi_{i1}\vee\ldots\vee \vphi_{ik}$ of length $k$ over the Boolean variables $x_1,\ldots,x_n$
	chosen uniformly at random from the set of all $(2n)^{km}$ possible such formulas.
Since the very beginning research on the random $k$-SAT problem has been driven by two hypotheses.
First, that for any $k\geq3$ there is a certain critical clause-to-variable density $\rk>0$, the {\em $k$-SAT threshold}, where the probability that the random formula
is satisfiable drops from almost $1$ to nearly $0$.
Second, that random formulas with a density close to but below $\rk$ are ``computationally difficult'' in some intuitive sense~\cite{Broder,Cheeseman, MitchellSelmanLevesque}.

While over the past 20 years there has been tremendous progress on the first hypothesis~\cite{nae,yuval,DSS,Kosta,HajiSorkin}, only few advances have been made on the second one.
One exciting insight came from physics~\cite{pnas,MPZ}.
Namely, according physics predictions, the geometry of the set of satisfying assignments $S(\PHI)$ undergoes a dramatic
change well below the satisfiability threshold.
Specifically, at a certain density $m/n$ the set $S(\PHI)$ breaks up into tiny, well-separated ``clusters'' \whp\
In fact, if we choose a satisfying assignment $\SIGMA$ uniformly at random, then \whp\ it will belong to a cluster with many ``frozen variables'', which
take the same truth value in {\em all} the satisfying assignments in that cluster.
Thus, the set $S(\PHI)$ is broadly resembles an error-correcting code, except that there is no simple underlying algebraic structure.
In effect, if, say, a local search algorithm attempts to find a satisfying assignment, it would apparently have to have the foresight to steer into one cluster
and get all its frozen variables right almost in one go.
This appears impossible without a survey of the ``global'' dependencies amongst the variables.

Many of the physics predictions on the geometry of the set $S(\PHI)$, including ``clustering'' and ``freezing'', have by now become rigorous theorems.
Moreover, the clause/variable density where clustering and freezing occur matches the density up to which algorithms are rigorously known to find
satisfying assignments, at least asymptotically for large enough clause lengths $k$.
To be precise, the $k$-SAT threshold is asymptotically equal to $m/n=2^k\ln 2-(1+\ln 2)/2+o_k(1)$,
	where $o_k(1)$ hides a term that tends to $0$ in the limit of large $k$~\cite{Kosta,DSS}.
By comparison, algorithms are known to find satisfying assignments up to $m/n=(1+o_k(1))2^k\ln k/k$~\cite{BetterAlg}.
Furthermore, for $m/n>(1+o_k(1))2^k\ln k/k$ clustering and freezing occur~\cite{Barriers, ACR, Molloy}.
Thus, one might expect that random formulas turn ``computationally difficult'' for densities almost a factor of $k$ below the $k$-SAT threshold.
Yet despite the structural results and the compelling intuitive picture drafted by the physics work, it has emerged to be remarkably difficult to actually
{\em prove} that these structural properties pose a barrier even for fairly simple satisfiability algorithms.

\subsection{The main result.}
In this paper we provide such a proof for \walksat, one of the simplest non-trivial satisfiability algorithms.
\walksat\ is a local search algorithm.
It starts with a uniformly random assignment.
So long as the current assignment fails to be satisfying, the algorithm chooses a random unsatisfied clause and flips the value assigned to a random variable in that clause.
That clause will thereby get satisfied, but other, previously satisfied clauses may become unsatisfied.
If after a certain given number $\omega$ of iterations no satisfying assignment is found, \walksat\ gives up.
Thus, the algorithm is one-sided: it may find a satisfying assignment but it cannot produce a certificate that a given formula is unsatisfiable.
The pseudocode is shown in Figure~\ref{Fig_Walksat};
for a formula $\Phi$ with $m$ clauses and $\sigma \in \Sigma$ we write $U_\Phi(\sigma)$ for the set of all indices $i\in[m]$ such that clause $\Phi_i$ is unsatisfied  under $\sigma$ and we let $\cU_\varPhi(\sigma)=|U_\varPhi(\sigma)|$ be the number of unsatisfied clauses.
\walksat\ is known to outperform exhaustive search by an exponential factor in the worst case and the procedure has been an ingredient
for some of the best algorithms for the $k$-SAT problem~\cite{DW,Hertli,HMS,HSSW,IT,PPSZ,Schoning}.

\begin{figure}
\begin{algorithm}
\walksat$(\Phi,\omega)$ \\
Input: \emph{A $k$-CNF $\Phi$ on $V$ and an integer $\omega>0$.} \\
Output: \emph{A truth assignment.}\\
\emph{\algstyle{
1.\qquad Choose an initial assignment $\sigma^{[0]}$ uniformly at random.\\
2.\qquad For $i=0, \ldots, \omega$ do\\
3.\qquad\qquad If $\sigma^{[i]}$ is a satisfying assignment output $\sigma^{[i]}$ and halt. \\
4.\qquad\qquad Choose $\varPhi_i\in U_\varPhi(\sigma^{[i]})$ and an integer from $1\leq j\leq k$ uniformly at random.\\
5.\qquad\qquad Obtain $\sigma^{[i+1]}$ from $\sigma^{[i]}$ by flipping the value of 
	the variable of the literal $\varPhi_{ij}$.\\
7.\qquad If $\sigma^{[\omega]}$ is a satisfying assignment output $\sigma^{[\omega]}$.
	Otherwise output `failure'.
}}
\end{algorithm}
\caption{The \walksat\ algorithm.}\label{Fig_Walksat}
\end{figure}

For a given formula $\Phi$ and $\omega>0$ we let $\suc(\varPhi,\omega)$ be the probability 
(over the random decisions of the algorithm only) that \walksat$(\Phi,\omega)$ will find a satisfying assignment.
Thus, $\suc(\vphi,\omega)$ is a random variable that depends on the random formula $\PHI$.

\begin{theorem} \label{theo_1}
There is exists a constant $c>0$ such that for all $k$ and all
 $m/n\geq c2^k\ln^2k/k$ \whp
 	$$\suc(\vphi,\lceil\exp( n/k^2)\rceil)\leq\exp(-n/k^2).$$
\end{theorem}

The random formula $\PHI$ is well-known to be unsatisfiable \whp\ if $m/n>2^k\ln2$.
Therefore, the condition $m/n>c2^k\ln^2k/k$ in \Thm~\ref{theo_1} implies a lower bound on  the clause length $k$ for which the statement is non-vacuous.
We have not tried to optimise the constant $c$.

The density required by \Thm~\ref{theo_1} exceeds the clustering/freezing threshold by a factor of $c\ln k$,
but still the $k$-SAT threshold is almost a factor of $k$ away.
Moreover, the theorem shows that \walksat\ fails in a dramatic way:
	on typical random formula $\PHI$ the success probability of \walksat\ is exponentially small, even if we run \walksat\ for an exponential number of rounds.
In particular, even if we restart \walksat\ any polynomial number of times from a new starting point the cumulative success probability of all trials
will remain exponentially small.

Why is it difficult to prove a result such as \Thm~\ref{theo_1} given what we know about freezing/clustering?
At the densities well below the $k$-SAT threshold like in \Thm~\ref{theo_1} we know that a uniformly {\em random} satisfying truth assignment
of the random formula $\PHI$ will lie in a ``frozen cluster'' \whp\
But there may very well exist unfrozen clusters;
	in fact, recent physics work suggests that there are exponentially many~\cite{BDSZ}.
Hence, because \walksat\ just aims to find a single satisfying assignment rather than to sample one uniformly at random,
the algorithm just needs to be lucky enough to find one weak, unfrozen spot, as it were.
In other words, we have to rule out the possibility that the algorithm somehow manages to home
	in on those spots where the ``barriers'' of the set $S(\PHI)$ are easily overcome.

But establishing such a statement is well beyond the standard arguments for analysing algorithms on random structures.
The main techniques such as the ``method of differential equations'' are suitable merely to trace algorithms for a small linear number of steps
and run into severe difficulties if the algorithm ever backtracks.
By construction, \walksat\ backtracks constantly (very likely many variables will likely be flipped more than once) and we actually need to follow the algorithm
for an {\em exponential} number of steps.
Hence, a different approach is needed.
\Sec~\ref{Sec_outline} provides a detailed outline of the proof of \Thm~\ref{theo_1}.

\subsection{Related work}
On the positive side, \walksat\ is known to find satisfying assignments for densities $m/n<2^k/(25k)$ for large enough $k$ in linear time~\cite{CF}.
Thus, the present paper matches the positve result up to a  $O_k(\ln^2k)$-factor.
Physics arguments suggest that \walksat\ should actually be effective up to $m/n=(1+o_k(1))2^k/k$~\cite{Monasson}, but not beyond.
Positive results for \walksat\ for small $k$ were obtained by Alekhnovich and Ben-Sasson~\cite{AB}.
{Additionally, they obtained exponential lower bounds for \walksat\ in the planted $3$-SAT problem for densities far above the satisfiabilty threshold, where in the planted model a random $3$-SAT formula is chosen conditioned on the existence of one solution~\cite{ABplanted}.}

Gamarnik and Sudan~\cite{GS2} obtained negative results for a class of algorithms that they call ``sequential local algorithms'' for the random $k$-NAESAT problem,
a cousin of random $k$-SAT.
Sequential local algorithms set the variables $x_1,\ldots,x_n$ of the random formula one by one in the natural order.
They do not backtrack.
The algorithm determines the value of variable $x_i$ based on the depth-$t$ neighborhood of $x_i$ in the hypergraph respresenting the formula.
To this end the algorithm takes into account the values assigned to those variables amongst $x_1,\ldots,x_{i-1}$ that occur in that part of the hypergraph.
The class of sequential local algorithms encompasses truncated version of message passing algorithms such as Belief Propagation Guided Decimation
and Survey Propagation Guided Decimation.
`Truncated' means that only a bounded number of parallel message updates are allowed; however, to reach an asymptotic fixed point of the messages it may be
necessary to update for $\Theta(\ln n)$ rounds.
The main result of~\cite{GS2} is that sequential local algorithms fail to find NAE-satisfying assignments for clause/variable densities above
$C2^k\ln^2k/k$ for a certain constant $C>0$. 

While \walksat\ is not a sequential local algorithm, we critically use one idea of the analysis from~\cite{GS2}, called ``overlap structures'' in that paper.
Specifically, Gamarnik and Sudan prove that for an appropriate integer $l$ no $l$-tuple of NAE-satsifying assignments exist with pairwise distance about $n\ln(k)/k$ if the clause/variable densities is above $C2^k\ln^2k/k$.
However, a coupling argument shows that if a local sequential algorithm were likely to succeed,
then there would have to be such an $l$-tuple of NAE-satisfying assignments with a non-vanishing probability.
Actually the idea of overlap structures originates from the work of Rahman and Virag~\cite{RV}, who improved the density of an earlier negative
result of Gamarnik and Sudan~\cite{GS1} for a more specialised class of algorithms for the independent set problem.
The definition of ``mists'' in the present paper is directly inspired by overlap structures. 

The first and the last author obtained negaitve results for message passing algorithms for random $k$-SAT that do not require bounds
on the number of iterations~\cite{BP,SP}.
Specifically, \cite{BP} shows that a basic version of Belief Propagation Guided Decimation fails to find satisfying assignments for densities
$m/n>C2^k/k$ for a certain constant $C>0$.
Moreover, \cite{SP} shows that a basic version of the conceptually more powerful Survey Propagation Guided Decimation algorithm
fails if $m/n>(1+o_k(1))2^k\ln k/k$.

Further negative results deal with DPLL-type algorithms.
In particular, Achlioptas, Beame and Molloy~\cite{ABM} proved that certain types of DPLL-algorithms fail for densities $m/n>C2^k/k$.
By comparison, unit clause propagation-type algorithms succeeds on random $k$-SAT formulas for $m/n<C'2^k/k$~\cite{ChaoFranco2,mick}.
Finally, the best current algorithm for random $k$-SAT succeeds for $m/n\leq(1+o_k(1))2^k\ln k/k$ but seems to fail beyond~\cite{BetterAlg}.

\subsection{Notation and preliminaries}
Throughout the paper we set $\rho=2^{-k}m/n$ and $\kappa=\ln k/k$.
We assume tacitly that $k,n$ are sufficiently large for our various estimates to hold.
Moreover, from here on out we always assume that $m/n>c2^k\ln^2k/k$.
(As mentioned above, the assumption that $k$ is large is justified because we assume that $m/n>c2^k\ln^2k/k$ and
the random formula is unsatisfied \whp\ if $m/n>2^k\ln2$.)

If $l$ is a literal, then we write $|l|$ for the underlying variable.
Thus, $|l|=x_i$ if $l=x_i$ or $l=\neg x_i$.
Moreover, the Hamming distance of two truth assignments $\sigma,\tau$ is denoted by $\dist(\sigma,\tau)$.
Additionally, for two truth assignments $\sigma,\tau:V\to\{0,1\}$ we let
	\begin{align}
	\Delta(\sigma,\tau)=\{x\in V: \sigma_1(x)\neq\tau(x)\}
	\end{align}
be the set of variables where $\sigma,\tau$ differ; hence, $|\Delta(\sigma,\tau)|=\dist(\sigma,\tau)$.
Further, for $\sigma\in\{0,1\}^n$ and $r_1,r_2\geq0$ define
\begin{align}
	\cD_\sigma(r_1,r_2)=\{\tau\in \Sigma: \lfloor r_1\kappa n \rfloor\leq\dist(\sigma,\tau)\leq \lfloor r_2\kappa n\rfloor\}.
\end{align} 
Hence, $\cD_\sigma(r_1,r_2)$ is a ring around $\sigma$ with inner radius $r_1\kappa n$ and outer radius $r_2\kappa n$.
Additionally, let $\cD_\sigma(r)=\cD_\sigma(r,r)$ be the set of assignments at distance exactly $r\kappa n$.

Recall that the {\em Kullback-Leibler divergence} of $p,q\in(0,1)$ is defined as
	$$\KL qp=q\ln\frac qp+(1-q)\ln\frac{1-q}{1-p}.$$
The following well-known lemma ``Chernoff bound'' states that the Kullback-Leibler divergence provides the rate function of the binomially distribution
(e.g., \cite[p.~21]{JLR}). 

\begin{lemma}\label{lem_chernoff}
Let $p,q\in(0,1)$ be distinct and let $X_n=\Bin(n,p)$.
Then
\begin{align*}
	\lim_{n\to\infty}\frac1n\ln\pr\brk{X\leq qn}&=-\KL{q}{p}&\mbox{if $q<p$},\\
	\lim_{n\to\infty}\frac1n\ln\pr\brk{X\geq qn}&=-\KL{q}{p}&\mbox{if $q>p$}.
\end{align*}
\end{lemma}

\noindent
We are going to need the following ``random walk'' version of \Lem~\ref{lem_chernoff}.

\begin{corollary}\label{cor_walk}
Suppose that $(W_n)_{n\geq1}$ is a sequence of independent random variables such that $0<\pr[W_n=1]=1-\pr[W_n=-1]=p<1/2$.
Let $q>0$.
Then
\begin{align*}
\lim_{n\to\infty}\frac1n\ln\pr\brk{\sum_{t=1}^nW_n\geq qn}&=-\KL{(1+q)/2}{p}.
\end{align*}
\end{corollary}
\begin{proof}
Let $X_t=(1+W_t)/2$ for all $t\geq1$.
Then $S_n=\sum_{t=1}^nX_t$ is a binomial random variable with parameters $n$ and $p$ and
$\sum_{t=1}^nW_t=2(\sum_{t=1}^nX_t)-n$.
Hence, $\sum_{t=1}^nW_t\geq qn$ iff $\sum_{t=1}^nX_t\geq n(1+q)/2$ and the assertion follows from \Lem~\ref{lem_chernoff}.
\end{proof}

\section{Outline}\label{Sec_outline}

\noindent
The classical worst-case analysis of \walksat\ goes as follows.
Suppose that $\Phi$ is a satisfiable $k$-SAT formula on $n$ variables and fix a satisfying assignment $\tau$.
At any step the algorithm flips a randomly chosen variable in an unsatisfied clause.
Because $\tau$ must satisfy that clause, there is at least a $1/k$ chance that the algorithm moves toward $\tau$.
Hence, in the case $k=2$ the distance evolves at least as good as in an unbiased random walk, and thus we expect to reach $\tau$
or another satisfying assignment in $O(n^2)$ steps~\cite{Papadimitriou}.
By contrast, for $k\geq3$ the corresponding random walk has a drift away from $\tau$ and the probability of reaching $\tau$
in polyonmial time from a random starting point is exponentially small.
Yet calculating the probability of starting at distance a bit less than $n/2$ from $\tau$ and then dashing towards it
reveals that \walksat\ beats the naive $2^n$ exhaustive search algorithm~\cite{Schoning}.

Of course, on a random formula this analysis is far from tight.
For example, for $m/n$ below the satisfiability threshold the number $|S(\PHI)|$ of satisfying assignments is typically exponential in $n$.
In fact, \whp\ we have $\ln|S(\PHI)|=n\ln2+\frac mn\ln(1-(1+o_k(1))2^{-k})$~\cite{ACR}.
Hence, if $m/n=O_k(2^k\ln^2k/k)$, then \whp\ the number of satisfying assignments is as large as $$|S(\PHI)|=2^{n(1-O_k(\ln^2k/k))}.$$
This observation obliterates some obvious proof ideas, such as combining the random walk argument from the previous paragraph
with some sort of a union bound on the number of satisfying assignments;
there is just too many of them.%
	\footnote{The second author's master thesis contained an argument based on combining the random walk analysis with a union bound.
		However, that argument requires that $m/n=(1+o_k(1))2^k\ln 2$, a much stronger assumption than that of \Thm~\ref{theo_1}.}

Another type of approach that seems doomed is meticilously tracing every step of the \walksat\ algorithm.
This is basically what the proof of the positive \walksat\ result from~\cite{CF} does.
Such analyses typically depend on the principle of deferred decisions, i.e., the idea that the parts of the formula that the algorithm has not inspected yet
are ``random'', subject to some relatively weak conditioning.
This kind of approach can follow an algorithm for a small linear number of steps.
But here we are trying to prove a statement about an exponential number of iterations.
By that time the algorithm will likely have visited every clause of the formula several times over and thus there is ``no randomness left''.
Hence, we need a different approach.

Our strategy is to split the analysis in two parts.
First, we are going to formulate a few {\em quasirandom properties}.
We will show that \walksat\ is exponential on {\em any} given formula that has these properties.
Second, we will prove that the random formula has these quasirandom properties \whp\
A similar type of argument was used, e.g., in prior work on message passing algorithms~\cite{BP, SP}.

The key is to come up with the right quasirandom properties.
To this end, we need to develop an intuition as to what \walksat\ actually does on a random input $\PHI$.
Because \walksat\ starts from a random assignment, initially there will be about $2^{-k}m=\rho n$ unsatisfied clauses.
In fact, we can establish a stronger, more geometric statement.
Let $T(\Phi)$ be the set of all truth assignments $\tau \in\{0,1\}^n$ such that $\cU_{\PHI}(\tau)\leq n\rho/10$
	(i.e., the number of violated clauses is a tenth of what we expect in a random assignment).
Set $\kappa=\ln k/k$.
Then a union bound shows that the initial assignment $\sigma^{[0]}$ will most likely be at distance at least $10\kappa n$ from all $\tau\in T(\PHI)\supset S(\PHI)$.

The second observation is that \walksat\ will likely have a hard time entering the set $T(\PHI)$.
Intuitively, for $m/n>(1+o_k(1))2^k\ln k/k$ it is not just the set $S(\PHI)$ that shatters into tiny well-separated clusters, but even the set $T(\PHI)$
has this property.
Moreover, the no man's land between different clusters provides no clues that nudge \walksat\ towards any one of them.
In fact, there is a repulsion effect.
To be precise, consider a ``target assignment'' $\tau\in T(\PHI)$ and suppose that $\sigma\in\{0,1\}^n\setminus T(\PHI)$
has distance at most $100\kappa n$ from $\tau$.
Because $\sigma\not\in T(\PHI)$, the assignment leaves at least $n\rho/10$ clauses unsatisfied.
Let us pretend that these unsatisfied clauses are random.
Then if we pick a variable in an unsatisfied clause randomly, the probability of hitting a variable in $\Delta(\sigma,\tau)$ is as small as $100\kappa<0.1$
	(for large enough $k$).
Hence, there is a 90\%\ chance that \walksat\ will move away from $\tau$, deeper into no man's land.
Thus, to reach a satisfying assignment or, in fact any assignment in $T(\PHI)$ \walksat\ would have to beat the odds and overcome a substantial negative drift, which is
exponentially unlikely.

However, there is one point that we missed.
Although the probability of walking towards one satisfying assignment at distance at most $100\kappa n$ from the present assignment may be small,
the total number of satisfying assignments is enormous and \walksat\ just has to find any one of them.
In other words, at any step \walksat\ may be taking part in an exponential number of ``lotteries''.
While any one of them may be rigged against the algorithm, the sheer number of simultaneous lotteries may yet give the algorithm a chance to succeed in polynomial time.

To rule this possibility out we introduce the concept of a {\em mist}, which is an adaptation of the ``overlap structures'' from~\cite{GS2}.
More precisely,  we will argue that we do not need to track the distance between \walksat's current assignment and the entire set $T(\PHI)$
but merely the distance to a much smaller set $\cM$ of assignments.
This subset is ``sparse'' in the sense that for any truth assignment $\sigma$ the number of assignments in $\cM$ at distance at most $10\kappa n$ from
$\sigma$ is bounded by $k$ rather than exponential in $n$.
We will use this fact to argue that at any time the algorithm only takes part in at most $k$ lotteries rather than an exponential number.
This will enable us to prove that reaching $T(\PHI)$ will most likely take an exponential amount of time.

Formally, let $\Phi$ be a $k$-CNF on the variable set $x_1,\ldots,x_n$.

A \emph{mist} of $\Phi$ is a set $\cM\subset T(\Phi)$ of assignments with the following two properties.
\begin{description}
\item [MI1] the assignments in $\cM$ have pairwise distance at least $2\kappa n$.
\item [MI2] for each $\sigma\in T(\Phi)$ there exists $\mu\in\cM$ such that $\dist(\mu,\sigma)\leq 2\kappa n$.
\end{description}

Thus, the points of the mist are spread out but there is one near every assignment in $T(\Phi)$.
Let
\begin{align*}
	\cD(\varPhi,\cM)=\bigcup_{\sigma\in\cM} \cD_{\sigma}(0,10)
\end{align*} 
be the set of all assignments at distance at most $10\kappa n$ from $\cM$.
Moreover, for a truth assignment $\sigma$ and a set $W\subset\{x_1,\ldots,x_n\}$  let
	\begin{align}\label{eqX}
	X_\varPhi(W,\sigma)=\sum_{i\in U_\Phi(\sigma)}\sum_{j\in[k]}\vecone\{|\Phi_{ij}|\in W\}
	\end{align}
be the number of occurrences of variables from $W$ in the unsatisfied clauses $U_\varPhi(\sigma)$.
Further, call $\Phi$ {\em quasirandom} if there is a mist $\cM$ such that the following three statements hold.

\begin{description}
 \item[\textbf{Q1}] 
 		we have $|\cD(\varPhi,\cM)|\leq 2^n\exp(-2n/k^2)$.
 \item[\textbf{Q2}] for any $\tau \in \{0,1\}^n$ we have $|\cM\cap\cD_\tau(0,10)|\leq k$.
 \item[\textbf{Q3}] for every $\mu\in\cM$ in the mist and each $\sigma \in \cD_{\mu}(0,100)\setminus T(\varPhi)$ we have
 	$X_\varPhi(\Delta(\mu,\sigma))\leq  k \cU_\varPhi(\sigma)/10$.
\end{description}

\noindent
Thus, the set $\cD(\varPhi,\cM)$ is small and thus it is exponentially unlikely for the initial random $\sigma^{[0]}$ to belong to this set.
Moreover, there are no more than $k$ elements of the mist $\cM$ in the vicinity of any one assignment $\tau$.
Finally, {\bf Q3} says that if $\tau\not\in T(\Phi)$ is an assignment with many unsatisfied clauses at distance no more than $100\kappa n$ from $\mu\in\cM$, then the probability that \walksat\ takes a step from $\tau$ towards $\mu$ does not exceed $10\%$.
Indeed, $X_\varPhi(\Delta(\mu,\tau))$ is the number of flips that take \walksat\ closer to $\mu$, and $ k \cU_\varPhi(\tau)$
is the total number of  possible flips.

Now, proving \Thm~\ref{theo_1} comes down to establishing the following two statements.

\begin{proposition}\label{prop_walk}
If $\Phi$ is quasirandom, then $\suc(\Phi,\lceil\exp( n/k^2)\rceil)]\leq \exp(-n/k^2)$.
\end{proposition}

\begin{proposition}\label{prop_quasi}
If $m/n\geq 195\cdot 2^k\ln^2k/k$, then $\vphi$ is quasirandom \whp
\end{proposition}

\noindent
We prove \Prop~\ref{prop_walk} in \Sec~\ref{Sec_walk} and \Prop~\ref{prop_quasi} in  \Sec~\ref{sec_quasi}.
\Thm~\ref{theo_1} is immediate from \Prop s~\ref{prop_walk} and~\ref{prop_quasi}.

\section{Proof of \Prop~\ref{prop_walk}}\label{Sec_walk}

\noindent
Suppose that $\Phi=\Phi_1\wedge\cdots\wedge\Phi_m$ is a quasirandom $k$-CNF on the variables $x_1,\ldots,x_n$.
Let $\cM$ be a mist such that {\bf Q1--Q3} hold and set $\omega=\lceil\exp(n/k^2)\rceil$.
Condition {\bf Q1} provides that the event $\cA=\{\sigma^{[0]} \notin \cD(\varPhi)\}$ has probability
	\begin{align}\label{eqPrA}
	\pr\brk\cA&\geq1-\exp(-2n/k^2).
	\end{align}
In the following we may therefore condition on $\cA$.

The key object of the proof is the following family of events: for $\mu\in\cM$ and $1\leq t_1< t_2\leq \omega$ let
	\begin{equation}\label{eqH}
	H_\mu(t_1,t_2)=\cbc{\dist(\sigma^{[t_1]},\mu)=\lfloor 10\kappa n\rfloor,\dist(\sigma^{[t_2]},\mu)=\lfloor5\kappa n\rfloor,
		\forall t_1\leq t\leq t_2:\sigma^{[t]} \in \cD_{\mu}(5,10)\setminus T(\varPhi)}.
	\end{equation}
In words, $H_\mu(t_1,t_2)$ is the event that at time $t_1$ \walksat\ stands at distance precisely $\lfloor 10\kappa n\rfloor$ from $\mu$,
 that the algorithm advances to distance $\lfloor5\kappa n\rfloor$ at time $t_2$ while not treading closer to $\mu$ but staying in $\cD_{\mu}(5,10)$ at any intermediate step, and that
\walksat\ does not hit $T(\varPhi)$ at any intermediate step.
Let
	$$\cH=\bigcup_{\mu\in\cM,0\leq t_1<t_2\leq\omega}H_\mu(t_1,t_2).$$

\begin{fact}\label{lem_hjii}
We have
	$\pr\brk{\exists t\leq\omega:\sigma^{[t]}\in S(\varPhi)|\cA}
		\textstyle\leq\pr\brk{\cH|\cA}$.
\end{fact}
\begin{proof}
Recall that $S(\varPhi)\subset T(\varPhi)$.
Suppose that $\sigma^{[t]}\in S(\varPhi)$ for some $t\leq\omega$; then the algorithm halts at time $t$.
Let $t_0< t$ be minimum such that $\sigma^{[t_0]}\in T(\Phi)$.
Then there exists $\mu\in\cM$ such that $\dist(\mu,\sigma^{[t_0]})<2\kappa n$.
Further, given $\cA$ we have $\dist(\sigma^{[0]},\mu)>10\kappa n$.
Hence, for some $0<t_1<t_0$ the event $\dist(\sigma^{[t_1]},\mu)\leq 10\kappa n$ occurs for the first time.
Moreover, there exists a minimum $t_2$ such that $t_1<t_2<t_0$ and $\dist(\sigma^{[t_2]},\mu)\leq 5\kappa n$.
Since \walksat\ moves Hamming distance one in each step, $H_\mu(t_1,t_2)$ occurs.
\end{proof}

\noindent
To show that $\cH$ is exponentially unlikely we are first going to estimate the probability of a single event $H_{\mu}(t_1,t_2)$.

\begin{lemma}\label{lem_drift}
Let $\tau_1 \notin T(\varPhi)$ and $\mu\in\cM$ be such that $\dist(\tau_1,\mu)=\lfloor10\kappa n\rfloor$.
Then
\begin{align*}
	\pr\brk{H_\mu(t_1,t_2)|\cA, \sigma^{[t_1]}=\tau_1} \leq \exp(-\kappa n/2)\quad\mbox{for all $1\leq t_1\leq t_2\leq \omega$}.
\end{align*} 
\end{lemma}
\begin{proof}
For an index $t_1< t\leq t_2$ 
define 
	\begin{align}\label{eqY}
	Y_{t+1}=\dist(\sigma^{[t+1]},\mu)-\dist(\sigma^{[t]},\mu)+2\cdot\vecone\{\sigma^{[t]}\not\in \cD_{\mu}(5,10)\setminus T(\Phi)\}.
	\end{align}
If the event $H_\mu(t_1,t_2)$ occurs, then $\sigma^{[t]}\in \cD_{\mu}(5,10)\setminus T(\varPhi)$ for all $t_1\leq t\leq t_2$ and
	$\sum_{t_1\leq t< t_2}Y_{t+1}\leq1-5\kappa n$.
Moreover, we claim that 
	\begin{align}\label{eqLem_Drift1}
	\Erw[Y_{t+1}-Y_t|\sigma^{[t]}\not\in T(\Phi)]&\geq 4/5.
	\end{align}
Indeed, at time $t+1$ \walksat\ chooses an unsatisfied clause and then a variable from that clause uniformly at random.
If $Y_{t+1}<Y_t$, then the chosen variable is from the set $\Delta(\mu,\sigma^{[t]})$ of variables where $\sigma^{[t]}$ and $\mu$ differ.
By (\ref{eqX}) the probability of this event equals $X_\Phi(W,\sigma^{[t]})/k\cU_\Phi(\sigma^{[t]})$.
Hence,  {\bf Q3} shows that the probability that $\dist(\sigma^{[t+1]},\mu)<\dist(\sigma^{[t]},\mu)$ is bounded by $0.1$, unless 
$\sigma^{[t]}\not\in \cD_{\mu}(5,10)\setminus T(\Phi)$.
Consequently, (\ref{eqLem_Drift1}) follows from the definition (\ref{eqY}).

If we let $(W_t)_{t\geq1}$ be a sequence of independent $\pm1$-random variables such that
$\pr[W_t=-1]=0.1$ and $\pr[W_t=1]=0.9$, then (\ref{eqLem_Drift1}) implies
	\begin{align*}
	\pr\brk{H_\mu(t_1,t_2)|\cA, \sigma^{[t_1]}=\tau_1}&\leq
		\pr\brk{\sum_{t_1\leq t< t_2}Y_{t+1}\leq1-5n\ln k/k}\leq
		\pr\brk{\sum_{t_1\leq t< t_2}W_{t}\leq1-5n\ln k/k}.
	\end{align*}
Thus, the assertion follows from \Cor~\ref{cor_walk} and the fact that $H_\mu(t_1,t_2)$ can occur only if $t_2-t_1\geq5\kappa n$,
because \walksat\ moves Hamming distance one in each step.
\end{proof}

\begin{proof}[Proof of \Prop~\ref{prop_walk}]
By Lemma \ref{lem_drift} each of the events contributing to $\cH$ occurs only with probability at most $\exp(-\kappa n/2)$ given $\cA$.
But since the number of assignments in the mist $\cM$ and hence the number of individual events $H_\mu(t_1,t_2)$
may be much larger than $\exp(n\kappa/2)$, a simple union bound on $\mu\in\cM$ won't do.
Indeed, the real problem here is the size of the mist and not the number of possible choices of $t_1,t_2$,
because $t_1,t_2\leq\omega$ and $\omega$ is (exponential but) relatively small.
In other words, we do not give away too much by writing
\begin{align}
\pr\brk{\cH|\cA}&
	\leq \sum_{0\leq t_1< t_2\leq\omega} \pr\brk{\textstyle\bigcup_{\mu\in\cM} H_\mu(t_1,t_2)\bigg|\cA}\nonumber\\
&=\sum_{0\leq t_1< t_2\leq\omega} \sum_{\sigma\in \Sigma}\pr\brk{\textstyle\bigcup_{\mu\in\cM} H_\mu(t_1,t_2)|\cA,\sigma^{[t_1]}=\sigma}
	\pr\brk{\sigma^{[t_1]}=\sigma|\cA}\nonumber\\
	&\leq \sum_{0\leq t_1< t_2\leq\omega} \max_{\sigma\in \Sigma}\pr\brk{\textstyle\bigcup_{\mu\in\cM} H_\mu(t_1,t_2)|\cA,
		\sigma^{[t_1]}=\sigma}\nonumber\\
&\leq \sum_{0\leq t_1< t_2\leq\omega} \max_{\sigma\in \Sigma}\sum_{\mu\in\cM}\pr\brk{ H_\mu(t_1,t_2)|\cA,\sigma^{[t_1]}=\sigma}.\label{theo_equ_2}
\end{align}
To bound the last term, we recall from (\ref{eqH}) that
	$\pr\brk{ H_\mu(t_1,t_2)|\cA,\sigma^{[t_1]}=\sigma}=0$ unless $\dist(\mu,\sigma)=\lfloor10\kappa n\rfloor$.
Hence, \textbf{Q2} implies that for any $\sigma\in\Sigma$ the sum on $\mu$ in (\ref{theo_equ_2}) has at most $k$ non-zero summands.
Therefore, Lemma \ref{lem_drift} gives
\begin{align}\label{theo_equ_1}
	\max_{\sigma\in \Sigma}\sum_{\mu\in\cM}\pr\brk{ H_\mu(t_1,t_2)|\cA,\sigma^{[t_1]}=\sigma}\leq k\exp(-n\kappa/2).
\end{align}
Plugging (\ref{theo_equ_2}) into (\ref{theo_equ_1}) and recalling the choice of $\omega$, we get
\begin{align}
\pr\brk{\cH|\cA}
&\leq \omega^2k\exp(-\kappa n/2)
\leq \exp(-n/k^2),\label{theo_equ_4}
\end{align}
with room to spare.
Finally, the assertion follows from (\ref{eqPrA}), Fact~\ref{lem_hjii} and (\ref{theo_equ_4}).
\end{proof}

\section{Proof of Proposition \ref{prop_quasi}}\label{sec_quasi}

\noindent
We begin with the following standard `first moment' bound.

\begin{lemma}\label{lem_exp_T}
We have $\Erw\abs{T(\vphi)}\leq 2^n\exp\br{-\rho n/2)}$.
\end{lemma}
\begin{proof}
For any fixed assignment $\sigma\in\{0,1\}^n$ the number $\cU_{\PHI}(\sigma)$ of unsatisfied clauses has distribution $\Bin(m,2^{-k})$.
Therefore, by \Lem~\ref{lem_chernoff} and our assumption on $m/n$,  
	\begin{align*}
	\pr\brk{\sigma\in T(\vphi)}&=\exp(-m\KL{0.1\cdot 2^{-k}}{2^{-k}}+o(n))\leq\exp(-\rho n/2).
	\end{align*}
Thus, the assertion follows from the linearity of expectation.
\end{proof}

To proceed, we construct a mist $\cM$ of the random formula $\vphi$ by means of the following iterative procedure.
\begin{enumerate}
\item Initially let $\cM=\emptyset$.
\item While $T(\vphi)\setminus\bigcup_{\mu\in\cM}\cD_{\mu}(0,2)\neq\emptyset$, add an arbitary element of this set to $\cM$.
\end{enumerate}
Let us fix any possible outcome $\cM$ of the above process.
Of course, $\cM$ depends on $\PHI$ but we do not make this explicit to unclutter the notation.
We now simply verify the conditions {\bf Q1--Q3} one by one.

\begin{lemma}\label{lem_Q1}
\textbf{Q1} holds with probability $1-\exp(-\Omega(n))$
\end{lemma}
\begin{proof}
We start with a naive bound on the number of assignments in $\cD_\sigma(0,10)$ centered at an arbitrary $\sigma \in \Sigma$.
Stirling's formula shows that for any fixed assignment $\sigma\in\{0,1\}^n$,
\begin{align*}
	|\cD_\sigma(0,10)|\leq\sum_{j\leq 10\kappa n}\bink{n}{j}\leq n\exp(10n\ln^2 k/k).
\end{align*} 
Hence, the construction of $\cM$ ensures that
	$|\cD(\vphi,\cM)|\leq |T(\vphi)|\cdot n\exp(10n\ln^2k/k).$
Thus,
\begin{align*}
	\Erw\brk{|\cD(\vphi)|} &\leq \Erw\brk{|T(\vphi)|}\cdot n\exp(10n\ln^2(k)/k).
\end{align*}
Consequently, the assertion follows from \Lem~\ref{lem_exp_T} and our assumption on $\rho$.
\end{proof}

For an assignment $\sigma\in \Sigma$ let $\cC(\sigma)$  be the set of all possible unsatisfied clauses under $\sigma$ on the variable set $x_1,\ldots,x_n$.
Then $|\cC(\sigma)|=n^k$ for all $\sigma\in\Sigma$.

The following Lemma proving that with high probability \textbf{Q2} holds in $\PHI$ is similar to the statement in~\cite{GS2} that certain ``overlap structures'' do not exist (where an ``overlap structure'' is an $l$-tuple of NAE-satisfying assignments with pairwise distance $\sim\kappa n$ for an appropriate integer $l$.)
This concept is an adaption of a bound on intersection densities for tuples of independent sets in sparse $d$-regular graphs from~\cite{RV}.
There it is shown that no tuple of large local independent sets intersecting each other in a certain way exists in a $d$-regular graph \whp\
We are going to prove a similar statement, namely that no $m$-tuple of assignments with a small number of unsatisfied clauses that have pairwise distance $\sim \kappa$ and are all contained in $\cD_\tau(0,10)$ for some $\tau\in\Sigma$ exist.
Following~\cite{GS2} we also use an inclusion/exclusion estimate, while here of course we are not focussing on satisfying assignments but on assignments with   a relatively small number of unsatisfied clauses.

\begin{lemma}\label{lem_Q2}
\textbf{Q2} holds \whp
\end{lemma}
\begin{proof}
We prove the statement by way of a slightly different random formula model $\PHI'$.
In $\PHI'$ each of the $(2n)^k$ possible clauses is included with probability $q=m/(2n)^k$ independently in a random order.
A standard argument shows that this model is essentially equivalent to $\PHI$.
To be precise, we claim that for any event $\cE$ we have
	\begin{align}\label{eqeq}
	\pr\brk{\PHI\in\cE}&\leq O(\sqrt n)\pr\brk{\PHI'\in\cE}+o(1).
	\end{align}
To see this, let $\cG$ be the event that $\PHI$ does not contain the same $k$-clause twice, i.e., $\PHI_i\neq\PHI_j$ for all $1\leq i<j\leq m$.
A simple union bound shows that $\pr\brk{\PHI\in\cG}=1-O(1/n)$.
Moreover, let $\vec m'$ be the total number of clauses of $\PHI'$.
The $\vec m'$ is a binomial variable with mean $m$ and Stirling's formula shows that $\pr\brk{\vec m'=m}=\Theta(n^{-1/2})$.
Thus,  (\ref{eqeq}) follows from the observation that the distribution of $\PHI'$ given $\vec m'=m$ coincides with the distribution of $\PHI$ given $\cG$.

Hence, we are going to work with the model $\PHI'$.
Let $\cM'$ be the mist constructed for $\PHI'$ by means of our above procedure.
Moreover, for $\tau\in \{0,1\}^n$  let $P(\tau)$ be the set of all $k$-tuples $(\sigma_i)_{i\in[k]}$ with the following two properties.
\begin{description}
\item[\textbf{P1}] $\sigma_i \in\cD_\tau(0,10)$ for all $i\in [k]$ and 
\item[\textbf{P2}] $\dist(\sigma_i,\sigma_j)\geq 2n\kappa$ for all $i\neq j$.
\end{description}
Then
	\begin{align}\label{prop_part_equ_0}
	|P(\tau)| \leq n\cdot\binom{n}{n10\ln(k)/k}^{k} \leq n\cdot\bcfr{\eul k}{10\ln k}^{10n\ln k}\leq \exp\br{10 n\ln^2k}.
	\end{align} 
Further, if $\sigma_1,\sigma_2\in\{0,1\}^n$ are assignments such that $\dist(\sigma_1,\sigma_2)\geq 2\kappa n$, then 
	the number of possible unsatisfied clauses under both $\sigma_1$ and $\sigma_2$ satisfies
\begin{align}\label{prop_part_equ_11}
	|C(\sigma_1) \cap C(\sigma_2)|=(n-\dist(\sigma_1,\sigma_2))^k\leq ((1-2 \ln(k)/k)n)^k \leq k^{-2}n^k;
\end{align}
this is because a clause that is unsatisfied under both $\sigma_1,\sigma_2$ must not contain any literals on which the two assignments differ. 
We are going to upper bound the probability that for $(\sigma_i)_{i\in[k]}\in P(\tau)$ assignment $\sigma_i$
renders at most $ \rho n/10 $ clauses of $\PHI'$ unsatisfied given that all of $\sigma_1,\ldots,\sigma_{i-1}$ do so any $i=1,\ldots,k$.
The probability that this event occurs is upper bounded by the probability that $\PHI'$ contains at most $\rho n/10$ clauses from the set
\begin{align}\label{prop_part_equ_12}
	C(\sigma_i|\sigma_1,\ldots,\sigma_{i-1})=C(\sigma_i) \setminus \bigcup_{j=1}^{i-1} C(\sigma_j).
\end{align}
The estimate (\ref{prop_part_equ_11}) and inclusion/exclusion 
yield
\begin{align*}
	|C(\sigma_i|\sigma_1,\ldots,\sigma_{i-1})|\geq n^k(1-(i-1)k^{-2}).
\end{align*}
Hence, if we let
$Z_i=\Bin(\lfloor  n^k(1-(i-1)k^{-2})\rfloor,q)$, then
\begin{align}\label{prop_part_equ_1}
	\pr\brk{\sigma_i\in T(\vphi')|\sigma_1,\ldots,\sigma_{i-1}\in T(\vphi')} \leq \pr\brk{Z_i\leq \rho n/10}
\end{align}
(this step required that the clauses of $\vphi'$ appear independently).
By the Chernoff bound, for $i\leq k$ we have 
\begin{align}\label{prop_part_equ_4}
	\pr\brk{Z_i\leq \rho n/10}&\leq\exp\br{-\rho n/15}
\end{align}
Consequently, \textbf{P2}, (\ref{prop_part_equ_1}) and (\ref{prop_part_equ_4}) yield for any $(\sigma_i)_{i\in[k]}\in P(\tau)$,
\begin{align}\label{prop_part_equ_5}
	\pr \brk{\sigma_1,\ldots,\sigma_k\in T(\vphi)}&=\prod_{i=1}^k \pr\brk{\sigma_i\in T(\vphi)|\sigma_j\in T(\vphi) \text{ for all } j<i}
		\leq \exp\br{-k\rho n/15}.
\end{align}

Further, let $Q(\vphi',\tau)$ be the set of all $k$-tuples $(\sigma_i)_{i\in[k]}\in P(\tau)$ such that $\sigma_1,\ldots,\sigma_k\in T(\vphi')$.
Then (\ref{prop_part_equ_0}) and (\ref{prop_part_equ_5}) imply
\begin{equation}\label{prop_part_equ_51}
\Erw\brk{Q(\vphi',\tau)}\leq\exp\brk{n\br{10\ln^2(k)-k\rho/15}}.
\end{equation}
Summing (\ref{prop_part_equ_51}) on $\tau\in\{0,1\}^n$ and using $\rho\geq 195\ln^2(k)/k$, we get
\begin{align}\label{eqSamSum}
	\sum_{\tau\in\{0,1\}^n}\Erw\abs{Q(\vphi',\tau)}
	\leq\exp\brk{n\br{(2+10)\ln^2(k)-13\ln^2(k)}}
	= \exp(-\Omega(n)).
\end{align}

Finally, assume that $\PHI'$ violates {\bf Q2}.
Then there is $\tau\in\{0,1\}^n$ such that $Q(\vphi',\tau)\neq\emptyset$, because our construction of $\cM'$ ensures that $\cM'\subset T(\PHI')$
and that the pairwise distance of assignments in $\cM$ is at least $2n\kappa$.
Consequently, (\ref{eqSamSum}) shows together with Markov's inequality that 
$\PHI'$ violates {\bf Q2} with probability at most $\exp(-\Omega(n))$.
Thus, the assertion follows by transferring this result to $\PHI$ via (\ref{eqeq}).
\end{proof}

\begin{lemma}\label{lem_Q3}
$\vphi$ satisfies \textbf{Q3} \whp
\end{lemma}
\begin{proof}
Let $\cP=\cP_{\PHI}$ be the number of pairs $(\sigma,\tau)\in \{0,1\}^n\times(\cD_\sigma(0,100)\setminus T(\vphi))$
	such that $X_\vphi(\Delta(\sigma,\tau))>k\cU_\vphi(\tau)/10$.
To estimate $\cP$ fix a pair $(\sigma,\tau)$ and let $\cP_\vphi(\sigma,\tau)$ be the event that $X_\vphi(\Delta(\sigma,\tau))>k\cU_\vphi(\tau)/10$.
If $\tau\in \cD_\sigma(0,100)\setminus T(\vphi)$, then $\tau$ leaves at least $\cU_{\PHI}(\tau)\geq\rho n/10$ clauses unsatisfied.
More precisely, given $\cU_{\PHI}(\tau)$ each unsatisfied clause consists of $k$ independent random literals that are unsatisfied under $\tau$.
Since $\cD_\sigma(0,100)$, for any one of the $k\cU_{\PHI}(\tau)$ underlying variables the probability of belonging to $\Delta(\sigma,\tau)$ equals
$\Delta(\sigma,\tau)/n\leq 100\kappa$.
Therefore, \Lem~\ref{lem_chernoff} shows that
	\begin{align}\label{eqPst}
	\pr\brk{\cP_\vphi(\sigma,\tau)}&\leq\pr\brk{\Bin(k|\cU_{\PHI}(\tau)|,\Delta(\sigma,\tau)/n)>k\cU_\vphi(\tau)/10}\leq\exp(-k\rho n/10).
	\end{align}

Summing (\ref{eqPst}) on $\sigma\in\{0,1\}^n$ and $\tau\in\cD_\sigma(0,100)$ and using our assumption on $\rho$, we get
\begin{align*}
	\Erw\brk{\cP}\leq\sum_{\sigma,\tau}\pr\brk{\cP_\vphi(\sigma,\tau)}\leq4^n\exp(-k\rho n/10)\leq 2^{-n}
\end{align*}
Thus, the assertion follows from Markov's inequality.
\end{proof}

\noindent
Finally, Proposition \ref{prop_quasi} follows directly from Lemma \ref{lem_Q1} to \ref{lem_Q3}.

\end{document}